\documentclass{amsart}
\usepackage{amsmath, amsfonts, amssymb, amsthm, amscd, eucal, hyperref, graphics, color}
\usepackage[all]{xy}
\usepackage[OT2,T1]{fontenc}

\newtheorem{theorem}{Theorem}[section]
\newtheorem{proposition}[theorem]{Proposition}
\newtheorem{lemma}[theorem]{Lemma}
\newtheorem {corollary}[theorem]{Corollary}
\theoremstyle {definition}
\newtheorem {definition}[theorem]{Definition}
\newtheorem {example}[theorem]{Example}
\theoremstyle {remark}
\newtheorem{remark}[theorem]{Remark}
\newtheorem{question}{Question}

\def\Rep{\operatorname{Rep}}

\def\Hom{\operatorname{Hom}}
\def\Ext{\operatorname{Ext}}
\def\Spec{\operatorname{Spec}}

\def\Im{\operatorname{Im}}

\def\Aut{\operatorname{Aut}}
\def\GL{\operatorname{GL}}
\def\Gal{\operatorname{Gal}}

\newcommand{\fX}{\ensuremath{\mathcal X}}
\newcommand{\fG}{\ensuremath{\mathcal G}}

\newcommand{\fV}{\ensuremath{\mathcal V}}

\newcommand{\fO}{\ensuremath{\mathcal O}}
\newcommand{\fR}{\ensuremath{\mathcal R}}

\newcommand{\bF}{\ensuremath{\mathbb F}}

\newcommand{\bQ}{\ensuremath{\mathbb Q}}
\newcommand{\bZ}{\ensuremath{\mathbb Z}}

\begin{document}

\title{Hodge cohomology of \'etale Nori finite vector bundles}
\author{\DJ o\`an Trung Cuong}

\address{\DH o\`an Trung Cuong,
FB Mathematik, Universit\"at Duisburg-Essen, 45117 Essen, Germany 
and Institute of Mathematics, 10307 Hanoi, Vietnam.}
\email{cuong.doan@uni-due.de}

\subjclass[2000]{14F05, 14J60, 20C99}
\keywords{Hodge cohomology, \'Etale Nori finite vector bundles,  Frobenius periodic vector bundles}
\thanks{This work has been supported by the SFB/TR 45 ``Periods,
moduli spaces and arithmetic of algebraic varieties''}

\begin{abstract}
\'Etale Nori finite vector bundles are those bundles
 defined by representations of a finite \'etale group scheme in the usual way. In this note we show that in many cases the dimensions of the Hodge cohomology groups of such a vector bundle and of a twist of it by an automorphism of the ground field are the same. This generalizes to the higher rank case the result of Pink-Roessler \cite[Proposition 3.5]{pr}. \vspace{.25cm}
\end{abstract}
\maketitle

\section{Introduction}
Let $X$ be a smooth geometrically connected projective variety over a perfect field $k$ with a $k$-rational point $x\in X(k)$. Let $V$ be a vector bundle on $X$. The Hodge cohomology groups with coefficients in $V$ are defined as
$$H^i_\mathrm{Hdg}(X, V):= \bigoplus_jH^{i-j}(X,V\otimes_{\fO_X}\Omega^j_{X/k})$$
and we set $h^i_\mathrm{Hdg}(V):=\mathrm{dim}_kH^i_\mathrm{Hdg}(X, V)$. In \cite{pr}, Pink and Roessler showed that if $\mathrm{char}(k)=0$ and $V=L$ is a torsion line bundle, that is, $L^{\otimes n}\simeq \fO_X$ for some natural number $n$ then $h^i_\mathrm{Hdg}(L)=h^i_\mathrm{Hdg}(L^{\otimes a})$ for $a$ prime to $n$. They also posed the following question for the positive characteristic case (see \cite[Conjecture 5.1]{pr}): Assume that $\mathrm{char}(k)=p>0$, $X$ is liftable over the ring $W_2(k)$ of $2$-Witt vectors, $\dim(X)\leqslant p$ and $L$ is a torsion line bundle of order $n$ on $X$. Then is it the case that $h^i_\mathrm{Hdg}(L)=h^i_\mathrm{Hdg}(L^{\otimes a})$ for $a$ prime to $n$, $i\in \mathbb Z$? In fact, when $(n, p)=1$, Pink and Roessler showed that the conjecture is true as an easy consequence of a result of Deligne-Illusie \cite[Lemme 2.9]{di}. It is also noticed in  \cite[Remark 10]{eo} that a Riemann-Roch calculation implies that the question has a positive answer for curves. The remaining case is when $p$ divides $n$ and $X$ has the dimension at least $2$. As far as we are aware of, the only result until now is a positive answer of Esnault-Ogus for $n=p$ and $X$ ordinary.

The aim of this note is to generalize the result of Pink-Roessler to the higher rank case. Recall that Nori \cite{no} defined a vector bundle $V$ to be finite if $V$ satisfies an equality $f(V)\simeq g(V)$ for some distinct polynomials $f, g$ whose coefficients are non-negative integers. We say a vector bundle $V$ is Nori finite if it is a subquotient of a finite direct sum of finite vector bundles. In that case, by taking direct sums, tensors, duals, subquotients, $V$ generates a neutral Tannaka category (see \cite[Chapter 1, Proposition 3.7]{no}). Due to a well-known theorem of Saavedra, this category together with the fiber functor $x$ is equivalent to the  representation category $\Rep_k(G)$ of the finite group scheme $G$ of tensor automorphisms of the fiber functor $x$, together with the fiber functor obtained as the forgetful functor $\mathrm{Rep}_k(G) \rightarrow \mathrm{Vect}_k$. $G$ is called the Tannaka group scheme of $V$. Torsion line bundles are exactly Nori finite vector bundles of $\mathrm{rank}\, 1$, and the Tannaka group scheme is $\mu_n$ where $n$ is the torsion order. The condition $n$ being prime to $p$ is equivalent to the underlying Tannaka group scheme being \'etale. In the present note we consider only Nori finite vector bundle with an \'etale Tannaka group scheme. The \'etale assumption is always satisfied in characteristic zero due to a result of Cartier (see \cite[11.4]{wa}) and the category generated by a Nori finite vector bundle is always semi-simple. The situation is more complicated in characteristic $p>0$. For higher rank Nori finite vector bundles, from Artin-Schreier theory it is known that the Tannaka group scheme $G$ can be a $p$-group scheme and still \'etale. However, in this case the category $\Rep_k(G)$ is no longer semi-simple, this causes a lot of difficulties in studying the corresponding bundle. Now assume $G$ is \'etale and $k$ is algebraically closed. The category $\mathrm{Rep}_k(G)$ is equivalent to the representation category $\mathrm{Rep}_k(G(k))$ of the abstract group $G(k)$ which carries a natural action of $\Aut(k)$. Let $\rho: G(k) \rightarrow \GL_r(k), r=\mathrm{rank}(V)$, be the representation corresponding to $V$ and $\sigma\in\Aut(k)$. We denote the vector bundle corresponding to $\sigma\circ\rho$ by $V_\sigma$. If $V$ is a line bundle then $V_\sigma$ is just a power of $V$. The following question is a generalization of Pink-Roessler's problem for higher rank vector bundles.
\begin{question}\label{star}
When $\mathrm{char}(k)=p>0$, we assume in addition that $X$ is liftable over the ring $W_2(k)$ of $2$-Witt vectors on $k$ and $\dim(X)\leqslant p$. For $\mathrm{char}(k)\geqslant 0$, is it true that $h^i_\mathrm{Hdg}(V)=h^i_\mathrm{Hdg}(V_\sigma)$?
\end{question}
In this note we give a positive answer to Question \ref{star} in certain cases. When $\mathrm{char}(k)=p>0$, one is able to give a bound for the Hodge numbers of every Frobenius pullback of a Nori finite vector bundle $V$. Consequently, any pullback of $V$ by a big enough power of the Frobenius always satisfies the equality in Question \ref{star} (Proposition \ref{prop1} and Corollary \ref{prop1-1}). As an application, we get an answer to Question \ref{star} if in addition the associated representation $\rho$ is realizable over a finite field, that is, if it is conjugated over $k$ to a representation on $\GL_r(\bF_q)$ for some power $q$ of $p$. This allows us to give an algebraic proof for a positive answer for Question \ref{star} in characteristic zero (Proposition \ref{prop5}). Note that this has been handled before by Sauter \cite[Satz 4.2.1]{sa}. In fact she gave two proofs generalizing the proofs for torsion line bundles of Pink-Roessler and Esnault, both rely on the comparison theorem between de Rham cohomology and Betti cohomology. In the proofs of the results above we use implicitly properties of Frobenius periodic vector bundles, that is, those bundles invariant under the action of some power of the Frobenius. Further properties of these bundles are explored in the second part of the note. The main result is a characterization of Frobenius periodicity in terms of the rationality of the associated representation of the monodromy group (Proposition \ref{prop3}). This characterization is then used to relate the existence of a class of Frobenius periodic vector bundles with the existence of $p$-quotients of the fundamental group of the variety (Proposition \ref{prop4}).\vspace{.25cm}

\noindent{\bf Acknowledgments.} I thank Kay R\"ulling for the simple proof of Lemma \ref{31} and other useful discussions. H\'el\`ene Esnault brought to the author the main question with careful explanation. I am deeply grateful to her and to Ph\`ung H\^o Hai for many discussions, for their encouragement and patience. In particular, I thank H\'el\`ene Esnault for allowing me to present her result as Proposition \ref{prop1} in this note.

\section{Hodge cohomology of \'etale vector bundles}

Let $k$ be a perfect field and $X$ be a geometrically connected projective variety over $k$ with a $k$-rational point $x\in X(k)$. Nori in \cite[page 80, Definition]{no} defined a vector bundle $V$ to be finite if there are two distinct polynomials $f, g$ whose coefficients are non-negative integers such that $f(V)\simeq g(V)$ as vector bundles. We will say a vector bundle $V$ is Nori finite if it is a subquotient of a finite direct sum of finite vector bundles. We denote still by $x$ the fiber functor at the point $x$ from the category of coherent sheaves on $X$ to the category of $k$-vector spaces. If $V$ is a Nori finite vector bundle, Nori proved that by taking direct sums, tensors, duals, subquotients, $V$ generates a $k$-linear abelian rigid tensor category (see \cite[Chapter 1, Proposition 3.7]{no}). Let $(V, x)$ be the pair consisting of this category and the fiber functor restricted on it, then by a well-known theorem of Saavedra, $(V, x)$ is equivalent to the representation category $\mathrm{Rep}_k(G)$ of a finite group scheme $G$ together with the forgetful functor $\mathrm{Rep}_k(G) \rightarrow \mathrm{Vect}_k$. This is called the Tannaka duality and $G$ is called the Tannaka group scheme of $V$. For more on Nori finite vector bundles and Tannaka duality we refer to \cite{no, ehs}.

\begin{definition} A vector bundle $V$ on $X$ is called \'etale Nori finite if it is Nori finite and the Tannaka group scheme is \'etale. 
\end{definition}
\begin{remark}\label{23} (i) We say that a vector bundle $V$ is \'etale trivializable if there is an \'etale covering $\pi: Y \longrightarrow X$ such that $\pi^*V$ is a trivial bundle. Following Nori \cite[Chapter 1]{no} (see also \cite[Section 2]{ehs}), $V$ is \'etale Nori finite if and only if it is \'etale trivializable.

\item(ii) Assume that $V$ is an \'etale Nori finite vector bundle with the \'etale Tannaka group scheme $G$. Let $V_1$ be a subquotient of $V$. Then $V_1$ is a Nori finite vector bundle which is an object of the category $(V, x)$ and $(V_1, x)$ is a subcategory of $(V, x)$. Via Tannaka duality, $(V_1, x)$ defines also a finite group scheme $G_1$ which is in fact a quotient of $G$ (see, for example, \cite[Appendix, Proposition 5]{no}). In particular, $G_1$ is \'etale and $V_1$ is \'etale Nori finite.
\end{remark}

We first note that the \'etale assumption on $G$ in Question \ref{star} is necessary, without this restriction the conclusion in Question \ref{star} should be at the opposite side as the following example shows.
\begin{example}
Let $X$ be a super singular elliptic curve over a field $k$ of characteristic $p>0$. Let $F_X$ be the Frobenius on $X$, then the induced map $F_X^*: H^1(X, \fO_X) \longrightarrow H^1(X, \fO_X)$ vanishes \cite[page 332]{ha}. Let $0\not=\alpha\in \Ext_{\fO_X}^1(\fO_X, \fO_X)\simeq H^1(X, \fO_X)$ be a cohomology class of an extension $V$ of $\fO_X$ by itself. Since $F_X^*(\alpha)=0$, the sequence $0 \longrightarrow \fO_X \longrightarrow F_X^*V \longrightarrow \fO_X \longrightarrow 0$ splits. It is proved by Mehta-Subramanian \cite[Section 2]{ms} that $V$ is Nori finite with a finite local group scheme. It is easy to show that the non-splitting of the sequence $0 \longrightarrow \fO_X \longrightarrow V \longrightarrow \fO_X \longrightarrow 0$ implies (in fact, is equivalent to) $\dim_k H^0(X, V)=\dim_k H^0(X, \fO_X)=1$ and $\dim_kH^0(X, F_X^*V)=2\dim_kH^0(X, \fO_X)=2$. Therefore, $h_\mathrm{Hdg}^0(V)<h_\mathrm{Hdg}^0(F_X^*V)$.
\end{example}

In positive characteristic, the most emphasized automorphism in Question \ref{star} is $\sigma=F^{n}$ a power of the Frobenius of $k$. In that case, it is easy to see that $V_\sigma\simeq (F_X^n)^*V$ where $F_X$ is the Frobenius morphism of $X$. The first attempt to answer Question \ref{star} is the following proposition due to Esnault, we thank her for allowing us to present it in this note.
\begin{proposition}\label{prop1} Let $k$ be an algebraically closed field of characteristic $p>0$ and $X$ be a geometrically connected projective variety over $k$. Let $V$ be an \'etale Nori finite vector bundle on $X$. The set $\{h^i_\mathrm{Hdg}((F_X^n)^*V): i\in \mathbb Z, n\geqslant 0\}$ is finite.
\end{proposition}
\begin{proof}
It suffices to show that for each $0\leqslant i\leqslant \dim(X)$, there is a constant bounding above $h^i_\mathrm{Hdg}((F_X^n)^*V)$ for all $n\geqslant 0$.  Fix a $k$-rational point of $X$. We know by Remark \ref{23}(i) that there is an \'etale covering $\pi: Y\longrightarrow X$ which trivializes $V$, that is, $\pi^*V\simeq \fO_Y^{\oplus r}, r=\mathrm{rank}(V)$. One gets a short exact sequence
$$0 \longrightarrow V \longrightarrow \pi_*\fO_Y^{\oplus r} \longrightarrow V_1 \longrightarrow 0.$$
Since $\pi_*\fO_Y$ is \'etale Nori finite, $V_1$ is also \'etale Nori finite by Remark \ref{23}(ii). Note that $(F_X)^* (\pi_*\fO_Y)\simeq \pi_*\fO_Y$. Hence by twisting the sequence with $\Omega^1_{X/k}$, one obtains
$$h^i_\mathrm{Hdg}((F_X^n)^*V)\leqslant rh^i_\mathrm{Hdg}(\pi_*\fO_Y)+h^{i-1}_\mathrm{Hdg}((F_X^n)^*V_1), \ \text{ for all } n\geqslant 0.$$
Using induction on $i$, the right hand side is bounded above by a constant not depending on $n$, therefore $h^i_\mathrm{Hdg}((F_X^n)^*V)$ is bounded above by a constant.
\end{proof}

By the proof of \cite[Lemme 2.9]{di} (see \cite[Lemma 11.1]{ev} for a precise statement), for each $i\leqslant p$ the sequence $\{h^i_\mathrm{Hdg}((F_X^n)^*V)\}_{n=0}^\infty$ does not decrease. So we get immediately the following consequence of Proposition \ref{prop1}.
\begin{corollary}\label{prop1-1}
 Let $X$ be a geometrically connected projective variety over an algebraically closed field $k$ of characteristic $p>0$. Let $V$ be an \'etale Nori finite vector bundle on $X$. For each $i\leqslant p$, $h^i_\mathrm{Hdg}((F_X^n)^*V)$ is a constant for $n\gg 0$.
\end{corollary}

\begin{definition}
 Let $H$ be a group and $\rho: H \longrightarrow \GL_r(k)$ be a representation of $H$. Let $k_0\subseteq k$ be a subfield. We say that $\rho$ is realizable over $k_0$ if it is conjugated over $k$ to a representation $\rho_0: H \longrightarrow \GL_r(k_0)$, that is, $\rho\simeq \rho_0\otimes_{k_0}k$.
\end{definition}

Let $V$ be an \'etale Nori finite vector bundle with a finite \'etale group scheme $G$. Over the algebraically closed field $k$, the representation category of $G$ is equivalent to the representation category of the abstract group $G(k)$. Let $\rho: G(k) \longrightarrow \GL_r(k)$ be the associated representation of $V$. For each $\sigma \in \Aut(k)$, denote by $V_\sigma$ the vector bundle associated to the representation $\sigma\circ\rho$. In the next proposition, we show that if $\rho$ is realizable over a finite field then $V$ satisfies the equality in Question \ref{star}. 

\begin{proposition}\label{prop2}
 Assume in addition that $X$ is smooth and liftable to the ring $W_2(k)$ of $2$-Witt vectors over $k$. Let $\sigma\in \Aut(k)$. If $\rho$ is realizable over $\overline{\mathbb F}_p$ then 
$$h^i_\mathrm{Hdg}(V)=h^i_\mathrm{Hdg}(V_\sigma), \ \text{ for all } i \leq p.$$
In particular, the equality holds for all $i\in \mathbb Z$ if $\dim X\leq p$.
\end{proposition}
\begin{proof}
 Let $F\in \Aut(k)$ be the Frobenius homomorphism. Assume that $\rho$ is conjugated to a representation $\rho_0: G(k) \longrightarrow \GL_r(\overline{\mathbb F}_p)$. Then $\sigma\circ\rho$ is conjugated to $\sigma\circ\rho_0$. Moreover, from the Galois theory $\sigma\mid_{\overline{\mathbb F}_p}=F^n$ is a power of the Frobenius for some $n>0$. Hence $\sigma\circ\rho_0=F^n\circ\rho_0$ and we can replace $\sigma$ by $F^n$. It then suffices to prove
$$h^i_\mathrm{Hdg}(V)=H^i_\mathrm{Hdg}((F_X^n)^*V), \ \text{ for all } i\leq p.$$
In fact, $\rho$ is realizable over a finite field $\mathbb F_{p^N}$ for some $N>0$. Thus $F^N\circ\rho=\rho$ and $(F_X^N)^*V\simeq V$. Now the conclusion follows from Corollary \ref{prop1-1}.
\end{proof}
\begin{corollary}\label{28}
Keep the assumptions in Proposition \ref{prop2}. Assume in addition that $G(k)$ has order prime to $p$. Then $h^i_\mathrm{Hdg}(V)=h^i_\mathrm{Hdg}(V_\sigma), \ \text{ for all } i \leq p$.
\end{corollary}
\begin{proof}
Since the order of $G(k)$ is prime to $p$, all representations of $G(k)$ are realizable over $\overline\bF_p$ and the assertion is from Proposition \ref{prop2}. 
\end{proof}

In the case of characteristic zero, Sauter in \cite[Satz 4.2.1]{sa} gave an affirmative answer to Question \ref{star} by using the comparison theorem between de Rham cohomology and Betti cohomology. Using Proposition \ref{prop2}, in the rest of this section we give an algebraic proof for this fact. We first recall Brauer's theory of decomposition map in representation theory of finite groups. Let $A$ be a Dedekind domain with the field of fractions $K$ of characteristic zero and a residue field $k$ of $\mathrm{char}(k)=p>0$. Assume that $G$ is a finite group scheme over $\Spec A$ whose Hopf algebra $A[G]$ is a finitely generated projective $A$-module. Let $G_K$ and $G_k$ denote respectively the generic fiber and the special fiber of $G$ at $k$. Assume in addition that $G_k$ is \'etale. Let $\Rep_K(G_K)$ (resp. $\Rep_k(G_k)$) be the category of finite dimensional representations of $G_K$ (resp. $G_k$) over $K$ (resp. $k$) and $\fR(G_K)$ (resp. $\fR(G_k)$) the corresponding Grothendieck group. By taking finite extensions of the fields if necessary, we can assume that $K$ and $k$ are big enough and $G_K$ and $G_k$ are defined by the abstract groups $G_K(K)$ and $G_k(k)$, that means, $\Gal(K)$ and $\Gal(k)$ act trivially on $G_K(K)$ and $G_k(k)$ respectively (see \cite[6.4]{wa}). If we denote by $\fR(G_K(K))$ and $\fR(G_k(k))$ the Grothendieck groups of the representation categories of $G_K(K)$ and $G_k(k)$ over $K$ and $k$ respectively, we have $\fR(G_K)\simeq \fR(G_K(K))$ and $\fR(G_k)\simeq \fR(G_k(k))$. In the theory of representations of finite groups, it is well-known that there is a ring epimorphism $d: \fR(G_K(K)) \longrightarrow \fR(G_k(k))$  which is called the decomposition map and is constructed as follows (see \cite[Theorem 33]{se}). Let $[V]\in \fR(G_K(K))$. Let $V_0$ be an $A$-lattice of $V$, that is, $V_0$ is a finitely generated $A$-submodule of $V$ and $V_0$ generates $V$ as a $K$-vector space. Replacing $V_0$ by the sum of its image under the actions of the elements of $G_K(K)$, we assume that $V_0$ is stable under the action of $G_K(K)$. Hence $V_0\otimes_Ak$ is a $G_k(k)$-module. From $V$ one can get different $V_0\otimes_Ak$ (even non-isomorphic) by choosing different $V_0$, but we get the same equivalent class $[V_0\otimes_Ak]\in \fR(G_k(k))$. This defines the map $d: \fR(G_K(K)) \longrightarrow \fR(G_k(k))$. We record some properties of this map in \cite{se}.

\begin{lemma}\cite[Theorem 33 and Proposition 43]{se}\label{41}
The map $d: \fR(G_K) \longrightarrow \fR(G_k)$ is a ring epimorphism. Moreover, if the order of $G_K(K)$ is prime to $p$, $d$ is an isomorphism.
\end{lemma}

On $\Rep_K(G_K(K))$ and $\Rep_k(G_k(k))$ there are natural actions of $\Aut(K)$ and $\Aut(k)$ which induce actions on the Grothendieck groups $\fR(G_K)\simeq \fR(G_K(K))$ and $\fR(G_k)\simeq \fR(G_k(k))$. The next lemma relates the action of $\Aut(K)$ on $\fR(G_k)$ via the decomposition map with the action of the Frobenius homomorphism on $\fR(G_k)$.

\begin{lemma}\label{42}
Let $n=\dim_KK[G_K]$ and $\xi\in A$ be a primitive $n$-th root of unity. Let $\sigma\in\Aut(K)$ and write $\sigma(\xi)=\xi^a$. Assume that $p\gg 0$ and $p\equiv a \mod n$. Then the action of $\sigma$ on $\fR(G_K)$ via the decomposition map coincides with the action of the Frobenius homomorphism $F$ of $k$, that is, $d\circ \sigma=F\circ d$.
\end{lemma}
\begin{proof}
The assumption implies $G_K(K)\simeq G_k(k)$, we denote this group by $G_0$. Let $\rho: G_0 \longrightarrow \GL_r(K)$ be a group homomorphism. Following \cite[Theorem 24]{se}, $\rho$ is realizable over $\mathbb Q(\xi)$, that is, there is a representation $\rho_0: G_0 \longrightarrow \GL_r(\bQ(\xi))$ such that $\rho=\rho_0\otimes_{\bQ(\xi)}K$. Fix a basis $e_1, \ldots, e_r$ of the vector space $\bQ(\xi)^r$ and denote
$$V_0=\sum_{g\in G_0}\sum_{i=1}^rg(e_i)\bZ[\xi].$$
$V_0$ is stable under the action of $G_0$ and by taking modulo $p$, one obtains a representation $\bar \rho_0: G_0 \longrightarrow \GL_r(\bF_p(\xi))$. Similarly, from the representation $\sigma\circ\rho$, we get a lattice $(V_0)_\sigma$ which induces a representation $\overline{\sigma\circ\rho_0}$. Put 
$\bar\rho=\bar\rho_0\otimes_{\bF_p(\xi)}k$. Clearly $d([\rho])=[\bar\rho]$. On the other hand, from the assumption on $p$, $\overline{\sigma\circ\rho_0}=F\circ\bar\rho_0$ where $F$ is the Frobenius homomorphism of $k$. Therefore, $d\circ\sigma([\rho])=d([\sigma\circ\rho])=[F\circ\bar\rho_0]=F\circ d([\rho])$.
\end{proof}

Turning back to \'etale vector bundles. Assume that $K$ is an algebraically closed field of characteristic zero. Let $X$ be a connected smooth projective variety over $K$ and $V$ be a Nori finite vector bundle on $X$. By Tannaka duality, $V$ corresponds to a representation of a finite group scheme $G$ which is always \'etale by a result of Cartier (see \cite[11.4]{wa}). Since $\Rep_K(G)$ is equivalent to $\Rep_KG(K)$, there is a representation $\rho: G(K) \longrightarrow \GL_r(K)$ corresponding to $V$ through these equivalences. For each automorphism $\sigma\in \Aut(K)$, we denote by $V_\sigma$ the vector bundle corresponding to the representation $\sigma\circ\rho$. With these assumptions, we have.
\begin{proposition}\cite[Satz 4.2.1]{sa}\label{prop5}
$h^i_\mathrm{Hdg}(V)=h^i_\mathrm{Hdg}(V_\sigma)$ for all $i\in \mathbb Z, \sigma\in \Aut(K)$.
\end{proposition} 
\begin{proof}
By a standard argument in algebraic geometry, there is a subfield $K_0\subset K$ which is a finite (possibly transcendental) extension of $\mathbb Q$ such that $X, V, G$ have models over $K_0$ (see \cite[Proof of Th\'eor\`eme 2.1]{di}). There is an open subset $\Spec A \subset \Spec \fO_{K_0}$ such that $\Spec A/\Spec \mathbb Z$ is smooth and there are smooth models $\fX, \fG$ defining over $\Spec A$ whose generic fibers are $X, G$ respectively. Moreover, localizing $A$ if necessary, we can assume that the $A$-modules $H^i_\mathrm{Hdg}(\fX, \fV)$ are locally free of constant rank over $A$, where $\fV$ is a model of $V$ on $\fX/\Spec(A)$ and we denote this rank by  $h^i_\mathrm{Hdg}(\fV)$ too,  (see \cite[Section II.5]{mu} or \cite[Section III.12]{ha}). Due to Dirichlet's theorem on arithmetic progressions, there are infinitely many prime numbers $p$ such that $p\equiv a\mod n$. Let $p$ be such a prime such that $p$ is not invertible in $A$. Localizing the ring $A$ at a minimal prime ideal containing $pA$, we get a $1$-dimensional regular local ring $A_0$. Note that $A_0$ and $A$ have the same field of fractions $K_0$ and the residue field $k$ of $A_0$ is of characteristic $p$. Further more, by base change one gets  $h^i_\mathrm{Hdg}(V)=h^i_\mathrm{Hdg}(\fV)=h^i_\mathrm{Hdg}(\fV\otimes_A K_0)$ and $h^i_\mathrm{Hdg}(\fV)=h^i_\mathrm{Hdg}(\fV\otimes_{A}A_0)=h^i_\mathrm{Hdg}(\fV\otimes_A k)$, for all $i\in\bZ$ (see \cite[Section II.5]{mu} or \cite[Section III.12]{ha}). So we can assume from the beginning that $A$ is a discrete valuation ring with the residue field $k$ of characteristic $p\gg0$. Denote the fibers of $\fX, \fV, \fG$ at $p$ by $X_k, V_k, G_k$. Since $p\gg 0$, we can assume also $G_k$ is \'etale over $k$.

With the notations before Lemma \ref{41}, there is a decomposition map $d: \fR(G_K) \longrightarrow \fR(G_k)$ (we use freely a finite extension of $K_0$ if necessary) which is in fact an isomorphism. Let $\bar\rho: G_k \longrightarrow \GL_r(k)$ be the representation defining $V_k$, then $d([\rho])=[\bar\rho]$. Let $F$ and $F_{X_k}$ be the Frobenius morphisms on $k$ and $X_k$ respectively. Using Lemma \ref{42}, we obtain $d([\sigma\circ\rho])=[F\circ\bar\rho]$. By the choice of $p$, the category $\Rep_kG_k$ is semi-simple. Combining these facts we see that $F_{X_k}^*V_k$ is isomorphic to the fiber at $p$ of $\fV_\sigma$.  Therefore, from Corollary \ref{28} we obtain 
$$h^i_\mathrm{Hdg}(V)=h^i_\mathrm{Hdg}(V_k)=h^i_\mathrm{Hdg}((V_\sigma)_k)=h^i_\mathrm{Hdg}(V_\sigma),$$
for all $i\in \bZ$.
\end{proof}

\begin{remark}\label{43} (i) It should be noted that the answer to Question \ref{star} is always positive at the zero level. Indeed, let $V$ be an \'etale Nori finite vector bundle and $\rho: G(k) \longrightarrow GL_r(k)$ be the associated representation of $V$. One gets
$$h^0_\mathrm{Hdg}(V)=\dim_k \Hom_{\fO_X}(\fO_X, V)=\dim_k \Hom_{G(k)}(k, \rho)=\dim_k (k^r)^\rho,$$
where in $\Hom_{G(k)}(k, \rho)$, $k$ is the trivial representation and $(k^r)^\rho$ is the invariant subspace of $k^r$ under $\rho$. Clearly, $(k^r)^\rho\simeq (k^r)^{\sigma\circ\rho}$ for any $\sigma\in \Aut(k)$. So, $h^0_\mathrm{Hdg}(V)=h^0_\mathrm{Hdg}(V_\sigma)$.

\item (ii) Using (i), the Riemann-Roch theorem on curves allow to give a complete answer to Question \ref{star} for this case: If $X$ is a smooth curve over $k=\bar k$ of any characteristic and $V$ is an \'etale Nori finite vector bundle on $X$, then $h^i_\mathrm{Hdg}(V)=h^i_\mathrm{Hdg}(V_\sigma)$ for all $i\in \bZ$ and $\sigma\in\Aut(k)$. Let $\check V$ be the dual vector bundle of $V$, one has $h^0_\mathrm{Hdg}(V)=h^0_\mathrm{Hdg}(V_\sigma)$ and $h^2_\mathrm{Hdg}(V)=h^0_\mathrm{Hdg}(\check V)=h^0_\mathrm{Hdg}(\check V_\sigma)=h^2_\mathrm{Hdg}(V_\sigma)$ by (i) and Serre duality theorem. Moreover, $h^1_\mathrm{Hdg}(V)=\chi(V)-\chi(V\otimes_{\fO_X}\Omega_{X/k}^1)+h^0(V)+h^0(\check V)=\chi(V)+h^0(V)+h^0(\check V)$. Note that $\deg(V)=0$ since $V$ is defined by the representation of a finite group (see also \cite[Chapter 1, Proposition 3.4]{no}). Then Riemann-Roch theorem implies 
$h^1_\mathrm{Hdg}(V)=h^1_\mathrm{Hdg}(V_\sigma)$.
\end{remark}

\begin{remark} One also might think about a converse of Question \ref{star}: Let $X$ be a smooth geometrically connected projective varitey over $k=\bar k$ of characteristic $p>0$ and $V$ be a Nori finite bundle on $X$. If $h^i_\mathrm{Hdg}(V)=h^i_\mathrm{Hdg}((F_X^n)^*V)$ for all $i\geqslant 0$ and $n>0$, is $V$ \'etale? If the Tannaka group scheme of $V$ is local, Mehta and Subramanian \cite[Section 2]{ms} showed that $(F_X^n)^*V$ is a trivial bundle for some $n>0$. Hence, if $V$ is not trivial, $h^0(V)<h^0((F_X^m)^*V)=\mathrm{rank}(V)$. Unfortunately, the answer for the question above is negative in general. For a counter example, let $X$ be a hyperelliptic curve. Take an extension $V$ of $O_X$ by it self such that the cohomology class in $\Ext^1_{\fO_X}(\fO_X, \fO_X)\simeq H^1(X, O_X)$ is $\alpha+\beta$, where $\alpha, \beta\not=0$,
$F_X^*(\alpha)=0$ and $F_X^*(\beta)=\beta$. Then $V$ is Nori finite but its Tannaka group scheme is neither \'etale nor local (see \cite{ms, ehs}). We have $h^0(V)=h^0((F^n)^* V)=1$ for all $n>0$. By the same argument as in Remark \ref{43}(ii) we obtain $h^i_\mathrm{Hdg}(V)=h^i_\mathrm{Hdg}((F^n)^*
V)$ for all $n>0, i\geqslant 0$.
\end{remark}


\section{Frobenius periodic vector bundles}
In this section we always assume that $k$ is an algebraically closed field of characteristic $p>0$ and $X$ is a geometrically connected projective variety over $k$ with a fixed $k$-rational point $x\in X(k)$. We denote by $F_X: X \longrightarrow X$ the Frobenius morphism of $X$. In the proof of Proposition \ref{prop2}, the key point is if the associated representation of the monodromy group is realizable over a finite field then the corresponding vector bundle via Tannaka duality is Frobenius periodic. This notion is defined bellow.

\begin{definition}
A vector bundle $V$ on $X$ is called Frobenius periodic if $(F_X^n)^* V\simeq V$ for some $n>0$.
\end{definition}
\begin{remark}\label{32}
 (i) Following Lange-Stuhler \cite[Satz 1.4]{ls}, a Frobenius periodic vector bundle is \'etale trivializable. Thus $V$ is \'etale Nori finite by Remark \ref{23}(i). Combining this with Remark \ref{23}(i) and using \cite[Satz 1.4]{ls} again, if $k$ is the algebraic closure of $\mathbb F_p$ then Frobenius periodicity is equivalent to being \'etale Nori finite.

\item(ii) There are examples of \'etale Nori finite vector bundles which are not Frobenius periodic. Of course, to search for such examples one needs to assume the field has non-zero transcendental degree over $\bF_p$. Laszlo's example in \cite[Before Theorem 1.2]{bd} and Brenner-Kaid's example \cite[Example 2.10]{bk} provide such kind of vector bundles.
\end{remark}

Using this remark we get an analog for vector bundles of the main result of \cite{bd}.
\begin{corollary}\label{33} 
Let $V$ be a vector bundle of rank $r$ on $X$. Assume that $V$ is trivialized by an \'etale Galois covering 
$\pi: Y \longrightarrow X$ with $\deg(\pi)$ prime to $p$. Then $V$ is Frobenius periodic.
\end{corollary}
\begin{proof}
$V$ is \'etale Nori finite by Remark \ref{32}(i). Let $G$ be the corresponding finite \'etale group scheme and $\rho: G(k) \rightarrow GL_r$ the representation corresponding to $V$. By Remark \ref{23}(ii), $G(k)$ is a quotient of the structure group $\Gamma$ of $\pi$, thus $p\nmid\, ^\#G(k)$. This implies the semi-simplicity of the category $\Rep_k(G(k))$. Since there are only finitely many representations of rank $r$ in $\Rep_k(G(k))$, it is clear that $(F^n)\circ\rho\simeq \rho$ for $F^n$ some power of the Frobenius of $k$, or equivalently, $(F_X^n)^*V\simeq V$.
\end{proof}

In the next, we will give a characterization of Frobenius periodicity of vector bundles in terms of the rationality of the representations of the monodromy group. We first need a lemma in $p$-linear algebra.
\begin{lemma}\label{31}
Let $M=(\lambda_{ij})_{r\times r}\in M_{r\times r}(k)$. Let $q\in \mathbb N$ such that $p\mid q$. Denote
$$\Omega=\{(x_1, \ldots, x_r)\in k^r: (x_1^q, \ldots, x_r^q)=(x_1, \ldots, x_r)M\}.$$
Then the cardinality $^\#\Omega\leqslant q^r$ and the equality occurs if and only if $M$ is invertible.
\end{lemma}
\begin{proof}
 Let $I$ be the ideal of $k[X_1, \ldots, X_r]$ spanned by  $X_i^q-\sum_{j=1}^r\lambda_{ji}X_j, \ i=1, \ldots, r$ and let $A=k[X_1, \ldots, X_r]/I$. Then $A$ is an Artin ring with $k$-basis $X_1^{\alpha_1}\ldots X_r^{\alpha_r}, 0\leq \alpha_1, \ldots, \alpha_r<q$, thus of $k$-dimension $q^r$. Moreover, one has $\Omega=\Spec(A)(k)$. Thus $^\#\Omega\leq q^r$ with equality precisely when $A/k$ is \'etale as $k$ is algebraically closed. Since
$$\Omega^1_{A/k}=(dX_1, \ldots, dX_n)A/(\sum_{j=1}^r\lambda_{ji}dX_j: i=1,\ldots, r),$$
$A/k$ is \'etale if and only if $M$ is invertible.
\end{proof}

\begin{proposition}\label{prop3}
Let $X$ be a geometrically connected projective variety over $k$ with a fixed $k$-rational point and $V$ a vector bundle of rank $r$ on $X$. Then $V$ is Frobenius periodic if and only if $V$ is \'etale Nori finite and the associated representation $\rho: G(k) \longrightarrow \GL_r(k)$ is realizable over $\overline{\mathbb F}_p$ (hence, over a finite field).
\end{proposition}
\begin{proof}
The sufficient condition is clear. Conversely, assume $V$ is Frobenius periodic, then $V$ is \'etale Nori finite by Remark \ref{32}(i). It remains to show that $\rho$ is realizable over $\overline{\mathbb F}_p$. Let $n\in \mathbb N$ such that $F^n\circ \rho\simeq \rho$, where $F\in \Aut(k)$ is the Frobenius. By definition there is an invertible matrix $M=(\lambda_{ij})_{r\times r}\in \GL_r(k)$ such that $F^n\circ\rho(g)=M^{-1}\rho(g)M, \forall g\in G(k)$. We will show the existence of a base change matrix $N=(x_{ij})_{r\times r}$ such that if one puts $\rho_1(g)=N\rho(g)N^{-1}$ for all $g\in G(k)$, then $F^n\circ\rho_1=\rho_1$. Obviously such an $\rho_1$ is realizable over $\mathbb F_{p^n}$. Denote $q=p^n$ and $N_n=(x_{ij}^q)_{r\times r}$. From the equations of $\rho$ and $\rho_1$ we have
$$N\rho(g) N^{-1}=N_n F^n\circ\rho(g)N_n^{-1}=(N_nM^{-1}) \rho(g) (N_nM^{-1})^{-1},\ g\in G(k).$$
Hence it is enough to show that the equation $N_n=NM$ has an invertible solution $N$. Equivalently, the set
$$\Omega=\{(x_1, \ldots, x_r)\in k^r: (x_1^q, \ldots, x_r^q)=(x_1, \ldots, x_r)M\}$$
generates the vector space $k^r$. By Lemma \ref{31}, $^\#\Omega=q^r$. So it suffices to show that for any hyperplane $H\ni 0$ in $k^r$, $^\#(\Omega\cap H)<q^r$. Without lost of generality, assume the defining equation of $H$ is
$x_r=\sum_{i=1}^{r-1}\lambda_ix_i$. Let $(x_1, \ldots, x_r)\in \Omega\cap H$. Then
$$(x_1^q, \ldots, x_r^q)=(x_1, \ldots, x_r)M,$$
$$x_r=\sum_{i=1}^{r-1}\lambda_ix_i.$$
This implies that $(x_1^q, \ldots, x_{r-1}^q)=(x_1, \ldots, x_{r-1})M_1$, where $M_1=(\lambda_{ij}+\lambda_i)_{r-1 \times r-1}$. By Lemma \ref{31} again, the number of solutions of this system of equations does not exceed $q^{r-1}$. Thus $^\#(\Omega\cap H)<q^r$ and $\Omega$ generates the vector space $k^r$. We can take a basis $e_1, \ldots, e_r\in \Omega$ to formulate the rows of $N$.
\end{proof}

\begin{remark} The second half of the proof of Proposition \ref{prop3} is in fact another proof of Lang-Steinberg theorem \cite{st} for the general linear algebraic group $\GL_r$ with the action of the Frobenius homomorphism. The proof above uses Lemma \ref{31}. As it is elementary, we leave it for sake of completeness in this note.
\end{remark}

In the rest of this section, we will relate Frobenius periodic vector bundles with $p$-quotients of the fundamental group of $X$. Let $\mathcal C^\text{\'et}(X)$ be the category of all \'etale Nori finite vector bundles on $X$. This is a $k$-linear abelian rigid category which is by Tannaka duality equivalent to the representation category of a profinite group scheme $\pi^\text{\'et}(X, x)$ (see \cite{ehs}). In fact, $\pi^\text{\'et}(X, x)(k)\simeq \pi_1(X, x)$ is the \'etale fundamental group of $X$ defined by Grothendieck. If $V$ is Nori finite with a finite \'etale group scheme $G$ then $G$ is a quotient of $\pi^\text{\'et}(X, x)$. We will say that a group (respectively, a group scheme) has a $p$-quotient if it has a non-trivial quotient which is a $p$-group (respectively, a $p$-group scheme). The next proposition is in the same line with \cite[Theorem 1.10 and begining of \S2]{gi}.

\begin{proposition}\label{prop4} Let $X$ be a geometrically connected projective variety over $k$ with a $k$-rational point $x\in X(k)$. The following conditions are equivalent:

\item(i) The \'etale fundamental group scheme $\pi^\text{\'et}(X, x)$ has no $p$-quotients.

\item(ii) The \'etale fundamental group $\pi_1(X, x)$ has no $p$-quotients.

\item(iii) The induced homomorphism $F_X^*: H^1(X, \fO_X) \rightarrow H^1(X, \fO_X)$ is nilpotent, that is, for each $\alpha\in H^1(X, \fO_X)$, there is some $n>0$ such that $(F_X^*)^n(\alpha)=0$.

\item(iv) Every Frobenius periodic vector bundle whose monodromy group is a $p$-group is trivial.
\end{proposition}
\begin{proof}
The equivalent of $(i)$ and $(ii)$ is clear.

Let $F_X^*: H^1(X, \fO_X) \rightarrow H^1(X, \fO_X)$ be the homomorphism induced from the Frobenius morphism $F_X$ of $X$. Following \cite[Page 143, Corollary]{mu}, there is a decomposition of $k$-vector spaces $H^1(X, \fO_X)=H_\text{\'et}\oplus H_{loc}$ such that $F_X^*$ is nilpotent on $H_{loc}$ and $H_\text{\'et}$ has a basis $e_1, \ldots, e_d$ such that $F_X^*(e_i)=e_i, i=1, \ldots, d$. Obviously, $F_X^*$ is not nilpotent if and only if $H_\text{\'et}\not=0$. In that case let $0\not=\alpha\in H_\text{\'et}$ such that $F_X^*(\alpha)=\alpha$. Since $H^1(X, \fO_X)\simeq \Ext_{\fO_X}^1(\fO_X, \fO_X)$, $\alpha$ is a cohomology class of an extension $V$ of $\fO_X$ by itself and $F_X^*V\simeq V$. By Remark \ref{32}(i), $V$ is an \'etale Nori finite vector bundle. Denote the corresponding representation of the monodromy group by $\rho: G(k) \longrightarrow \GL_2(k)$. Note that $V$ is a non-trivial extension of $\fO_X$ by itself, there is a basis of $k^2$ such that 
\[\rho(g)=\left(\begin{matrix} 
1&a_g \\ 
0&1
\end{matrix}\right)\]
for all $g\in G(k)$, where $a_g\in k$ and $a_g\not=0$ for some $g$. This implies that $\Im(\rho)$ is a $p$-group. Since $\Im(\rho)$ is a quotient of $\pi_1(X, x)$, this proves the implication $(ii) \Rightarrow (iii)$.

For $(iii) \Rightarrow (iv)$, assume $V$ is a non-trivial Frobenius periodic vector bundle such that the monodromy group is a $p$-group. Put $r=\mathrm{rank}(V)$ and let $\rho: G(k) \longrightarrow \GL_r(k)$ be the representation corresponding to $V$. By Proposition \ref{prop3}, $\rho$ is realizable over $\overline{\mathbb F}_p$, that is, there is a basis of $k^r$ such that $\rho(g)=(a_{ij}^g)_{r\times r}\in \GL_r(\overline{\mathbb F}_p), \ g\in G(k)$. Note that $G(k)$ is a $p$-group, it is unipotent (see, for example, \cite[Proposition 26]{se}), that is, we can choose a basis such that $\rho(g)$'s are upper-triangle matrices with $1$ on the diagonal. Then one can choose a non-trivial representation $\rho^\prime: G(k) \longrightarrow \GL_2(k)$ which is a subquotient of $\rho$ and which is realizable over $\overline{\mathbb F}_p$. This $\rho^\prime$ defines a vector bundle $V^\prime$ which is also Frobenius periodic by Proposition \ref{prop3}. Obviously $V^\prime$ is an extension of $\fO_X$ by itself and if $\alpha\in H^1(X, \fO_X)$ is its cohomology class then $\alpha\not=0$ and $(F_X^*)^n(\alpha)=\alpha$ for some $n>0$. This shows that $F_X^*$ is not nilpotent on $H^1(X, \fO_X)$ and $(iii) \Rightarrow (iv)$.

Finally, if $H$ is a non-trivial $p$-quotient of $\pi_1(X, x)$ then it has a non-trivial representation over $\mathbb F_p$, for example, the regular representation. This representation defines a vector bundle $V$ which is Frobenius periodic by Proposition \ref{prop3}. This contradicts to the condition in $(iv)$. Then $(iv)$ implies $(ii)$.
\end{proof}


\end{document}